\documentclass[10pt,a4paper,leqno]{amsart}

\usepackage{amsmath}
\usepackage{amssymb}
\usepackage{amsthm}
\usepackage[pagebackref=false]{hyperref}
\usepackage{bm}
\usepackage[cal=boondoxo,calscaled=1]{mathalpha}

\usepackage{color}

\definecolor{darkgreen}{rgb}{0.5,0.25,0}
\definecolor{darkblue}{rgb}{0,0,1}
\definecolor{answerblue}{rgb}{0,0,0.75}

\hypersetup{colorlinks,breaklinks,
            linkcolor=darkblue,urlcolor=darkblue,
            anchorcolor=darkblue,citecolor=darkblue}

\newcommand{\pd}{\partial}

\newcommand{\Ex}{\mathbb{E}}
\renewcommand{\ss}{\bm{\sigma}}
\renewcommand{\d}{\mathrm{d}}

\newcommand{\loc}{{\rm loc}}

\newcommand{\KK}{\bm{\Sigma}}
\newcommand{\JJ}{\bm{j}}

\newcommand{\todelta}{\overset{\delta\downarrow 0}{\longrightarrow}}

\newcommand{\R}{\mathbb{R}}
\newcommand{\T}{\mathbb{T}}

\newcommand{\abs}[1]{\left | #1 \right |}
\newcommand{\norm}[1]{\left \| #1 \right \|}
\newcommand{\bk}[1]{ \left(  #1 \right)}

\allowdisplaybreaks

\theoremstyle{theorem}
\newtheorem{theorem}{Theorem}[section]
\newtheorem{lemma}[theorem]{Lemma}
\theoremstyle{remark}
\newtheorem{remark}{Remark}[section]

\title[Commutator estimates for gradient-noise SPDE]
{Second order commutator estimates in renormalisation theory for SPDEs with gradient-type noise}

\author[P.H.C. Pang]{Peter H.C. Pang}
\address[P.H.C. Pang]{Department of Mathematics\\
   University of Oslo\\
  NO-0316 Oslo\\ Norway}
\email{\href{mailto:ptr@math.uio.no}{\tt ptr@math.uio.no}}

\makeatletter
\@namedef{subjclassname@2020}{%
  \textup{2020} Mathematics Subject Classification}
\makeatother


\thanks{
The author is supported 
by the Research Council of Norway  project {\em INICE} (301538).}

\begin{document}

 \begin{abstract}
An important step in standard 
renormalisation arguments involve 
convolution against a standard mollifier. 
As pointed out in \cite{PS2018}, this 
generates second order commutator 
terms in equations with gradient-type 
noise. These are commutators 
similar to commutators in the 
well-known ``folklore lemma" of 
Di Perna--Lions \cite[Lemma II.1]{DL1989}, 
but not covered by standard 
renormalisation theory. 
In this note we establish the 
vanishing of these commutators 
for gradient-type noises on $\T^d$ 
not necessarily possessing divergence-free 
structure. 
 \end{abstract}

\maketitle

\section{Gradient type noises}

In this note we consider stochastic partial 
differential equations (SPDEs) with 
gradient type noise of the form:
\begin{align}\label{eq:ito_main0}
0 &= \d u + F[u]\,\d t + \nabla  \bk{\ss u}\circ \d W, 
	\qquad (t,x) \in (0,T) \times \T^d. 
\end{align}
We assume that 
$\ss \in W^{2,2pq/(p - q)}(\T^d;\R^{d \times m})$, 
for some fixed positive integer $m$. 
For convenience, we also fix $ p \ge 2$ and $1 \le q \le p$ throughout. 
Moreover $W = (B_1, \ldots, B_m)^\top$ is an 
$m$-tuple of independent, standard Brownian 
motions on a filtered probability space 
$(\Omega, \mathcal{F}, 
	\{\mathcal{F}_t\}_{t \in [0,T]}, \mathbb{P})$. 

Writing $\ss  = (\sigma_{ik})_{1 \le i \le d, 1 \le k \le m}$,  
we mean by our notation above the following:
$$
 \nabla  \bk{\ss u}\circ \d W
 	 =   
 	 	\pd_{x^i}   \bk{\sigma_{ik} u}\circ \d B_k, 
$$
where repeated indices are summed over 
appropriate ranges throughout this paper, i.e., over $1 \le i \le d$ 
and $1 \le k \le m$ above. 
To avoid potential ambiguities, 
where $f$ is scalar-valued, 
we define $\nabla \bk{\ss f} 
	= \pd_{x^i } \bk{\sigma_{ik} f}$, 
a $\R^m$-valued object --- an important 
case being $f \equiv 1$; and where $f$ is $\R^m$-valued, 
$\nabla \bk{\ss f} 
	 = \pd_{x^i} \bk{\sigma_{ik} f_k}$, 
a scalar-valued object. We shall use 
``${\rm grad}(f)$" to denote the gradient 
of a function $f$. 

Equation \eqref{eq:ito_main0},  
as written, motivates the 
study of the formally equivalent 
formulation with It\^o noise:
\begin{align}\label{eq:ito_main1}
0 = \d u + F[u]\,\d t + \nabla   \bk{\ss u} \cdot \d W 
	- \frac12 \nabla   \bk{\ss \nabla   \bk{\ss u}}\,\d t.
\end{align}

Under relatively mild assumptions, 
distributional solutions to the 
equation with Stratonovich noise 
are distributional solutions to the corresponding equation 
with It\^o noise (see, e.g., \cite[Section 2]{AF2011}).

Probabilistically strong solutions $u$ 
are usually required to satisfy the SPDE 
weakly in space, be predictable in time, 
and take values in a space $\mathcal{X}$.
$F$ is a possibly nonlinear map from 
$\mathcal{X}$ (as opposed to a map from $\R$, where $u$ takes values). A motivating example here 
is the stochastic Camassa--Holm 
equation (derived in \cite{Hol2015} on $\R$) 
in which $d = 1$, $u = \pd_x v$, and for every 
$\varphi \in C^1(\T)$,
\begin{equation*}
\begin{aligned}
\int_{\T} \varphi F[u]\,\d x  \,\d t
	&= - \int_{\mathbb{T}} \pd_x\varphi\, uv - \varphi\bk{P - v^2 - \frac12 u^2}\,\d x\,\d t ,
	\,\,\,\, \mathbb{P}-a.s.,\\
	 \bk{1 - \pd_{xx}^2} P  &= v^2 + \frac12 u^2.
\end{aligned}
\end{equation*}
%
%

This note studies an aspect of 
the renormalisation theory for \eqref{eq:ito_main1}, 
as it pertains to the gradient 
type stochastic ($\ss$-related) terms 
of \eqref{eq:ito_main1}, namely, 
$$
\nabla   \bk{\ss u} \d W 
	- \frac12 \nabla   \bk{\ss  \nabla  \bk{\ss u}}\,\d t 
= \pd_{x^i} \bk{\sigma_{ik} u} \,\d B_k  
	- \frac12 \pd_{x^j} \bk{\sigma_{jk}  
		\,\pd_{x^i} \bk{\sigma_{ik} u}}\,\d t.
$$

We use the label "renormalisation'' 
broadly, to mean the study of limits 
of $S_\ell(u_n)$, where  $\{S_\ell \}_{\ell \in \mathbb{N}} 
	\subseteq C^{1,1}(\R)$ is 
a family of suitably smooth and sufficiently 
slowly growing entropy functions, and 
$u_n$ is a sequence of approximate 
solutions tending to $u$ a.s., 
e.g., in some weak topology. 

In order to study $S(u)$, one often needs 
to derive the equation for $\d S(u)$ from 
\eqref{eq:ito_main1} (or $\d S(u_n)$ from 
the approximating system for $\d u_n$).
For stochastic equations, therefore, 
renormalisation is 
intimately related to deriving an 
It\^o formula for processes taking 
values in infinite dimensional spaces. 
There are well-known results for infinite 
dimensional It\^o formul\ae, e.g., in \cite{DZ2014,Kry2013} (see also references there). 
They are unavoidable for hyperbolic 
SPDEs, generally understood weakly. 
Such derivations 
sometimes involve convolution against 
a mollifier, so that the equation 
can be understood $x$-pointwise, and 
a standard It\^o formula applied. 
This is also the strategy we follow here. 

Since mollification and multiplication 
(by $\ss$, say) do not commute, there 
are commutator brackets that must 
be shown to vanish in appropriate topologies 
as the mollifier tends to a Dirac mass. 
These commutators are the focus of 
this note, and their vanishing are 
the results we prove 
in Lemmas \ref{thm:commutator1} and
\ref{thm:commutator2}, and 
Theorem \ref{thm:commutator3}
below, following \cite[Proposition 3.4]{PS2018}, 
\cite[Section 7]{HKP2023} quite closely.

First order commutator terms arise 
from mollification of gradient 
terms of the form  $b \cdot  \nabla u$ 
for $b \in L^1([0,T];W^{1,\alpha}(\T^d))$, 
in the work of Di Perna and Lions 
(see, e.g., \cite[Lemma II.1]{DL1989}, 
    \cite[Lemma 2.3]{Lio1996}). 
For gradient-type noise, it turns out that 
one encounters second order 
commutator terms as a result 
of the It\^o-to-Stratonovich 
conversion in a natural way. 
This was  pointed out in 
\cite{PS2018} in the context 
of the stochastic Boltzmann equation 
for $m = 1$, and a divergence-free 
condition on the $d$-vector $\ss$, 
which also satisfied
$\ss \in W^{1,2p/(p - 2)}_{df, \loc}(\R^d;\R^d)$ 
(with the subscript $df$ denoting 
divergence-free fields), 
$\ss \cdot {\rm grad}(\ss) \in W^{1,p/(p - 1)}_\loc(\R^d;\R^d)$and 
$u \in L^p(\Omega \times [0,T]; (L^p \cap L^1)(\R^d))$. 
Most of the computations 
are directly inspired by 
\cite[pp 654 -- 657]{PS2018}. 
The $d = m = 1$ analogue on $\T$ 
(obviously without the divergence-free 
condition) was worked out in 
\cite{GK2021,HKP2021,HKP2023}. 

Here we present results 
for which no algebraic 
conditions are assumed on the 
derivatives of $(\sigma_{ij})_{1 \le i \le d, 1 \le j \le m}$. 
These commutator estimates 
are not inherently stochastic, 
but gradient type 
noise force us to confront them, 
and similar transport/continuity 
type terms the the non-martingale 
parts of the equation (say, where  
$F[u] = b \cdot \nabla u$) do not. 
They constitute one small 
novel challenge of many that arise 
in the analysis of familiar PDEs 
perturbed by noise (compare, e.g., 
\cite{GHKP2022} and the corresponding 
results of \cite{CHK2005,XZ2002} 
in the deterministic setting). 

We end this introduction by 
pointing out that the relevance 
of gradient-type noises (sometimes 
called "convective noise" or 
"transport noise") 
are interesting for geometric mechanical 
reasons \cite{CH2018,Hol2015}, 
and have also been widely studied  
in stochastic fluid models and 
particularly in studies on 
regularisation by noise 
(see e.g., \cite{AF2011,Fla2010,FL2020}, and references there).

\section{Double commutator estimates}

In this section, we present our main results, 
on the vanishing of commutators that arise 
in the renormalisation of \eqref{eq:ito_main1} 
as it pertains to the noise terms. 
First we explain in detail how it 
is that certain new terms (vis-\`a-vis 
renormalisation in the deterministic context)
arise (see Eqs.\eqref{eq:ito_rule_applied} -- \eqref{eq:E2E3}). 
Next, we show in Theorem \ref{thm:commutator3} 
and lemmas leading to it, how 
these new terms have a double 
commutator structure, 
and also how they vanish as the 
mollification parameter tends to nought.

We begin by introducing  the mollifier. 
Let $J_\delta$ be a standard 
Friedrichs mollifier on $\T^d$. 
For $f \in L^p(\T^d; \R^m)$, write 
$f_\delta$ for the $m$-tuple $f * J_\delta$.
Mollifying the equation by integrating 
against $J_\delta(x - y)\,\d x$, we have
\begin{equation}\label{eq:ito_main_delta}
\begin{aligned}
0 = \d u_\delta + & F[u_\delta]\,\d t 
    + \nabla   \bk{ \ss  u_\delta}\cdot  \d W\\
   &   - \frac12 \nabla   \bk{ \ss    
             \nabla \bk{\ss  u_\delta}}\,\d t
        + E^{(1)}_\delta \,\d t
        + E^{(2)}_\delta  \cdot\d W 
        + E^{(3)}_\delta\,\d t ,
\end{aligned}
\end{equation}
where
\begin{equation}\label{eq:commutator_defin}
\begin{aligned}
E^{(1)}_\delta[u] 
&:=    F[u]*J_\delta 
    -     F[u_\delta],\\
E^{(2)}_\delta[u] 
&:=  \nabla   \bk{\ss  u}*J_\delta 
    - \nabla  \bk{\ss  u_\delta},\\
E^{(3)}_\delta[u] 
&:= -\frac12\Big(\nabla   \bk{\ss   
     \nabla   \bk{\ss   u}}\Big)*J_\delta 
    +  \frac12\nabla   \bk{\ss   
     \nabla   \bk{\ss   u_\delta}}.
\end{aligned}
\end{equation}
We expect these error terms 
introduced by mollification, to vanish 
in appropriate senses as $\delta \downarrow 0$.

Recall our convention that 
$\nabla \bk{\ss f} = \pd_i \bk{\sigma_{ik} f}$ 
is an $\R^m$-valued object when $f$ is scalar-valued (including $f \equiv 1$), 
and that $\nabla\bk{\ss f} = \pd_i \bk{\sigma_{ik} f_k}$ is scalar-valued  
when $f$ is $\R^m$-valued. 
Introducing an operator 
notation to emphasise the 
structure of the errors $E^{(i)}_\delta$, 
let $\KK f := \nabla   \bk{\ss  f}$, 
and $\JJ_\delta f := J_\delta * f$. 
Using commutator brackets, we can write 
\begin{equation}\label{eq:E_as_operators}
\begin{aligned}
E^{(2)}_\delta[u] 
&    = \big[ \JJ_\delta, \KK \big](u) 
   := \JJ_\delta \KK u - \KK \JJ_\delta u,\\
E^{(3)}_\delta[u] &= \frac12 \Big(\KK \KK \JJ_\delta u 
    - \JJ_\delta \KK \KK u \Big) = \frac12
	\Bigl( \KK \bigl[\KK, \JJ_\delta\bigr]( u)
	+\bigl[\KK, \JJ_\delta\bigr] \KK ( u) \Bigr).
\end{aligned}
\end{equation}

Recall we have fixed  $p \ge 2$. 
Suppose it is known {\em a priori} 
that $u$ is controlled thus:
\begin{align}\label{eq:apriori_u}
\Ex \norm{u}_{L^p([0,T]\times \T^d)}^p  \lesssim 1.
\end{align}

We can apply the It\^o formula  
to \eqref{eq:ito_main_delta}, understanding 
that equation $x$-{\em pointwise}. 
Recall that  $q \le p$. 
We use an entropy $S \in C^2(\R)$ 
satisfying the growth conditions 
\begin{align}\label{eq:bds_on_S}
\abs{S(r)} \lesssim 1 + \abs{r}^q, 
\quad \abs{S'(r)} \lesssim 1 + \abs{r}^{q - 1},
\quad \abs{S''(r)} \lesssim 1 + \abs{r}^{q - 2},
\end{align}
obtaining:
\begin{equation*}
\begin{aligned}
0  = \d S(u_\delta) & + S'(u_\delta) F[u_\delta]\,\d t 
    + S'(u_\delta) \nabla  \bk{\ss u_\delta  } \cdot  \d W\\
        & - \frac12 S'(u_\delta) \nabla  \bk{\ss  
               \nabla \bk{\ss  u_\delta}}\,\d t
        + S'(u_\delta)  E^{(1)}_\delta \,\d t\\
        &+ S'(u_\delta) E^{(2)}_\delta \cdot \d W 
        + S'(u_\delta) E^{(3)}_\delta - \frac12 S''(u_\delta) \abs{E^{(2)}_\delta + \nabla  \bk{\ss u_\delta } }^2\,\d t.
\end{aligned}
\end{equation*}

We assume no a priori bounds 
on $u$ beyond \eqref{eq:apriori_u}. 
In particular, we assume no bounds on ${\rm grad}(u)$.
Therefore we perform further manipulations 
on the mollified equation 
allows us to put it into a 
more "conservative form" (see, e.g., \cite[Eq.~(3.4)]{GHKP2022} in the stochastic Camassa--Holm context): 
\begin{equation}\label{eq:ito_rule_applied}
\begin{aligned}
0  &= \d S(u_\delta) + S'(u_\delta) F[u_\delta]\,\d t \\
&\quad\,\,    + \bk{\nabla \bk{\ss S'(u_\delta) } -  \nabla \ss \bk{S(u\delta) - S'(u_\delta) u_\delta)}} \cdot  \d W\\
    &\quad\,\,    +  \frac12 \nabla \bk{\ss S'(u_\delta) \nabla\bk{\ss u_\delta}} \,\d t 
    + \frac12 \nabla \bk{\ss \bk{S(u_\delta) - S'(u_\delta) u_\delta} \nabla \ss } \,\d t\\
        & \quad\,\,    - \frac12 S''(u_\delta) \abs{\nabla \ss}^2 u_\delta^2 \,\d t
         - \frac12 \bk{S(u_\delta) - S'(u_\delta) u_\delta} \nabla \bk{\ss \nabla \ss }\,\d t 
                 \\
        &\quad\,\,    + S'(u_\delta)  E^{(1)}_\delta \,\d t+ S'(u_\delta) E^{(2)}_\delta \cdot \d W 
        + S'(u_\delta) E^{(3)}_\delta\\
       & \quad\,\,     -  S''(u_\delta) \bk{\frac12 \abs{E^{(2)}_\delta}^2 + E^{(2)}_\delta \cdot \nabla  \bk{\ss u_\delta } }\,\d t ,
\end{aligned}
\end{equation}

We are interested in 
showing that the terms 
$E^{(i)}_\delta$ vanish 
appropriately as $\delta \downarrow 0$. 
In the following we
derive bounds for terms involving 
$E^i_\delta$ for $i = 2,3$ only. 
These are the terms that arise 
from the gradient noise. 

The bound for $E^{(2)}_\delta$ follows readily from 
Di Perna--Lions's folklore lemma \cite[Lemma II.1]{DL1989}. 
\begin{lemma}[Commutator estimates]\label{thm:commutator1}
Fix $p \ge 2$ and $1 \le q \le p$. Let $u \in L^p(\Omega \times [0,T]\times \T^d)$ 
and suppose $\ss \in W^{1,pq/(p - q)}(\T^d;\R^{d\times m})$, 
with $\ss \in W^{1,\infty}(\T^d;\R^{d\times m})$ if $p = q$. 
Define the commutator 
$E^{(2)}_\delta[u]$  as in 
\eqref{eq:commutator_defin}. 
The following convergence holds:
\begin{align*}
 E^{(2)}_\delta[u]
	\todelta 0 \qquad \text{
in $L^q(\Omega \times [0,T]\times \R)$}.
\end{align*}
\end{lemma}
\begin{remark}
It is only required that 
$\ss \in W^{1,pq/(p - q)}(\T^d;\R^{d\times m})$ 
here, but in dealing with the double 
commutator below, we shall be 
requiring the full assumption that 
$\ss \in W^{2,2pq/(p - q)}(\T^d;\R^{d\times m})$, 
\end{remark}
\begin{remark}
Recall the {\em a priori} bound 
\eqref{eq:apriori_u} and the 
growth conditions \eqref{eq:bds_on_S} on $S$. 
Given $p/(p - q + 1) \le q$ 
(because $q \mapsto p/(p - q + 1)$ is convex for $1 \le q \le p$, 
and equality is attained at the end-points), 
Lemma \ref{thm:commutator1} is enough 
to ensure the vanishing  
of terms in \eqref{eq:ito_rule_applied} 
involving $E^{(2)}_\delta$ as $\delta \downarrow 0$, except the term 
$\frac12 S''(u_\delta) E^{(2)}_\delta \cdot \nabla  \bk{\ss u_\delta }$. 
\end{remark}

The bound for $E^{(3)}_\delta$ is more delicate, 
and it is possible only to prove the convergence 
of the combination 
\begin{align}\label{eq:E2E3}
 S'(u_\delta) E^{(3)}_\delta -  S''(u_\delta) 
        		E^{(2)}_\delta \cdot \nabla \bk{\ss u_\delta} .
\end{align}
This combination has a double 
commutator structure. Therefore we 
first state the following technical 
lemma, whose proof we relegate 
to Section \ref{sec:proof_lemma2}.
\begin{lemma}[Double commutator estimate]\label{thm:commutator2}
Fix $p \ge 2$ and $1 \le q \le p$. 
Let $u \in L^p(\Omega \times [0,T] \times \T^d)$,
and suppose $\ss \in W^{2,2pq/(p - q)}(\T^d;\R^{d\times m})$, 
with $\ss \in W^{2,\infty}(\T^d;\R^{d\times m})$ if $p = q$. 
We have the convergence 
\begin{align*}
	\Bigl[\bigl[\KK,\JJ_\delta\bigr],\KK\Bigr](u)
	\todelta 0 \qquad \text{
in $L^q(\Omega \times [0,T]\times \R)$}.
\end{align*}
\end{lemma}

These estimates are sufficient 
to imply the convergence of the 
more complicated combinations 
of $E^{(2)}_\delta$ and $E^{(3)}_\delta$ 
in \eqref{eq:ito_rule_applied}, 
which is the main object of 
our investigation here (cf. {\cite[Proposition 7.4]{HKP2023}}).
\begin{theorem}
    [It\^{o}--Stratonovich related error term]
\label{thm:commutator3}
Fix $p \ge 2$ and $1 \le q \le p$. Let $S \in C^2(\R)$ 
satisfy the growth conditions \eqref{eq:bds_on_S}.

Let $u \in L^p(\Omega\times [0,T] \times \T^d)$, 
and suppose $\ss \in W^{2,pq/(p - q)}(\T^d;\R^{d\times m})$, 
with $\ss \in W^{2,\infty}(\T^d;\R^{d\times m})$ if $p = q$. 
Define $E^{(2)}_\delta$ and $E^{(3)}_\delta$ 
as in \eqref{eq:commutator_defin}. 
For each $\varphi \in C^\infty([0,T]\times \T^d)$, 
the following convergence holds:
\begin{equation}\label{eq:commutator_converge2}
	\begin{aligned}
\Ex \int_0^T\biggl| \, \int_\T 
		\varphi \bk{ S'(u_\delta)  E^{(3)}_\delta 
		-  S''(u_\delta) E^{(2)}_\delta \cdot \nabla  \bk{\ss u_\delta }}
		\,\d x\biggr| \,\d t\todelta 0.
	\end{aligned}
\end{equation}
\end{theorem}

\begin{proof}
Both the theorem statement 
and the calculations to follow 
take much inspiration from the proof
of \cite[Prop.~3.4]{PS2018}. However, whereas 
the commutator between the operators
$\tilde{\KK} f:= \ss \cdot \nabla f$ and
${\JJ_\delta} f$, was considered there (with $m = 1$)  
we have to consider the analogous question
for $\KK f := \nabla( \ss f)$ and $\JJ_\delta$.
We now explain how it is that \eqref{eq:E2E3} 
has a double commutator structure. 

%

Using the expression  
\eqref{eq:E_as_operators} 
for $E^{(2)}_\delta$, and 
following the calculations
in \cite[p.~655]{PS2018},
\begin{align*}
&	 E^{(2)}_\delta \cdot \nabla  \bk{\ss u_\delta }\\
	& = S''( u_\delta) u_\delta
	\bigl[\KK,\JJ_\delta\bigr]( u) \cdot \nabla \ss
	+ \nabla S'( u_\delta) \cdot \ss
	\bigl[\KK,\JJ_\delta\bigr]( u)
	\\ &= S''( u_\delta) u_\delta
	\bigl[\KK,\JJ_\delta\bigr]( u) \cdot \nabla \ss
	+\nabla \bigl(\ss  S'( u_\delta)
	\bigl[\KK,\JJ_\delta\bigr]( u)\bigr)
	- S'( u_\delta) \nabla \bigl(\ss
	\bigl[\KK,\JJ_\delta\bigr]( u)\bigr).
\end{align*}

Writing the final term as 
$- S'(u_\delta) \KK \bigl[\KK,\JJ_\delta\bigr]( u)$, 
we can add the above to $S'(u_\delta) E^{(3)}_\delta $ 
using  \eqref{eq:E_as_operators}, 
to get that:
\begin{align*}
& S'( u_\delta)\, E^{(3)}_\delta -S''( u_\delta)
	 E^{(2)}_\delta \cdot \nabla  \bk{\ss u_\delta }\\
	 &
	= \frac{1}{2}S'( u_\delta)
	\Bigl[\bigl[\KK,\JJ_\delta\bigr],\KK\Bigr]( u)	
	+S''( u_\delta) u_\delta
	\bigl[\KK,\JJ_\delta\bigr]( u) \cdot \nabla \ss 
		+\nabla \bigl(\ss  S'( u_\delta)
	\bigl[\KK,\JJ_\delta\bigr]( u)\bigr).
\end{align*}
For the term $\nabla \big(\ss S'( u_\delta)
	\bigl[\KK,\JJ_\delta\bigr]( u)\bigr)$, we integrate-
by-parts in $x$ against $\varphi$.
We know already that
$\bigl[\KK,\JJ_\delta\bigr]( u) =  E^{(2)}_\delta \todelta 0$
in $L^q(\Omega\times [0,T]\times \T)$ by Lemma
\ref{thm:commutator1}. 
Convergence of the double commutator
bracket is given by Lemma~\ref{thm:commutator2}.
Now the entire claim \eqref{eq:commutator_converge2}
follows from the assumption $u \in 
L^p(\Omega \times [0,T] \times \T^d)$ 
and from \eqref{eq:bds_on_S}.
\end{proof}




\section{Proof of Lemma \ref{thm:commutator2}}
\label{sec:proof_lemma2}

This section is solely devoted to the 
proof of Lemma \ref{thm:commutator2}. 
Nevertheless, we point out here that the 
technicalities of this proof 
constitute one of the primary departures 
from the $1$-dimensional or the 
divergence-free $d$-dimensional case.

\begin{proof}

Unpacking the commutator brackets, 
we have:
\begin{equation}\label{eq:commutators}
\begin{aligned}
	\Bigl[
	\bigl[\KK,\JJ_\delta\bigr],\KK\Bigr](u)
&	= \bigl[\KK,\JJ_\delta\bigr](\KK u)
	-\KK\bigl[\KK,\JJ_\delta\bigr](u)\\
 &
	= 2 \KK \JJ_\delta \KK  u-\JJ_\delta \KK \KK  u
	-\KK \KK \JJ_\delta  u.
\end{aligned}
\end{equation}
Implicitly summing over repeated indices 
over appropriate ranges, term-by-term we have:
\begin{align}
	&2\KK\JJ_\delta\KK  u(x)\notag\\
	& = 2  \sigma_{ik}(x) \int_{\T^d} 
		\pd_{x^ix^j}^2  J_\delta(x - y) \sigma_{jk}(y) u(y)\,\d y
	\label{eq:commute1}
	\\ & \quad
	+ 2\pd_{x^i} \sigma_{ik}(x) \int_{\T^d} 
		\pd_{x^j}  J_\delta(x - y) \sigma_{jk}(y) u(y)\,\d y,
	\label{eq:commute2}
	\\
	& \JJ_\delta \KK \KK  u (x)\notag\\
	& = \int_\R \pd_{x^ix^j}^2 J_\delta(x-y)
	\sigma_{ik}(y) \sigma_{jk}(y)  u(y) \;\d y
	\label{eq:commute3}
	\\ & \quad
	- \int_\R \pd_{x^i} J_\delta(x-y)\sigma_{jk}(y)
	\pd_{y^j} \sigma_{ik}(y) u(y)\;\d y,
	\label{eq:commute4}
	\\ \intertext{and}
	&\KK\KK \JJ_\delta  u (x)\notag\\
	& = \pd_{x^i} \big(\sigma_{ik}(x)
			\pd_{x^j} \sigma_{jk}(x)\big)
	\int_\R  J_\delta(x-y)   u(y) \;\d y
	\label{eq:commute5}
	\\ & \quad +2\sigma_{ik}(x)\pd_{x^j} \sigma_{jk}(x) 
	\int_\R \pd_{x^i} J_\delta(x-y) u(y)\;\d y
	\label{eq:commute6}
	\\ & \quad
	+ \sigma_{ik}(x) \sigma_{jk}(x)
	\int_\R \pd_{x^ix^j}^2	J_\delta(x-y) u(y)\;\d y
	\label{eq:commute7}
	\\ & \quad +\sigma_{ik}(x) \pd_{x^i} \sigma_{jk}(x)
	\int_\R \pd_{x^j} J_\delta(x-y) u(y)\;\d y
	\label{eq:commute8}.
\end{align}

We will estimate \eqref{eq:commute1} to
\eqref{eq:commute8}
by considering the sums
$$
{I}_1:= \eqref{eq:commute2}
-\eqref{eq:commute6},
\quad
I_2 := -\eqref{eq:commute4} -\eqref{eq:commute8},  
\quad
{I}_3:= \eqref{eq:commute1}
-\eqref{eq:commute3} - \eqref{eq:commute7},
$$
and the stand-alone integral \eqref{eq:commute5}. 
The terms \eqref{eq:commute2}, \eqref{eq:commute5}, 
and \eqref{eq:commute6} are, of course, 
absent when $\pd_j \sigma_{jk} \equiv 0$.

From \eqref{eq:commutators},
we see that
\begin{align}\label{eq:commutator_decomposition}
\Bigl[\bigl[\KK,\JJ_\delta\bigr],\KK\Bigr]( u)
	={I}_1+ I_2+{I}_3 -\eqref{eq:commute5}.
\end{align}
We will use \cite[Lemma II.1]{DL1989}
to establish that \eqref{eq:commutator_decomposition}
tends to zero in an appropriate sense.
Estimating the terms in \eqref{eq:commutator_decomposition}
separately, we have
{
\begin{align*}
&\norm{{I}_1}_{L^q(\T^d)}\\
	 &
	=2\biggl\|\int_{\T^d} \pd_{x^j} \sigma_{jk}(\cdot) \, \pd_{x^i}  J_\delta(\cdot-y)
	\Bigl( \sigma_{ik}(y) 
	- \sigma_{ik}(\cdot) \Bigr)
	 u(y) \;\d y\biggr\|_{L^q(\T^d)}
	\\ & 
 	\le 2 \biggl\| \int_{\T^d}
 	\abs{\cdot-y}\abs{\pd_{x^i} J_\delta(\cdot - y)}
 		\frac{\abs{\sigma_{ik}(\cdot)  - \sigma_{ik}(y) }}{\abs{\cdot - y}}
 			\abs{\pd_{x^j} \sigma_{jk}(\cdot)} \abs{ u(y)} \;\d y\biggr\|_{L^q(\T^d)}\\
& \le C \norm{\abs{\cdot} \pd_{x^i} J_\delta}_{L^1(\T^d)} 
	\norm{\abs{\nabla \sigma_{ik}}^2}_{L^{pq/(p - q)}(\T^d)} 
		\norm{u}_{L^p(\T^d)}\\
& \le C \norm{u}_{L^p(\T^d)},
\end{align*}
}where we have used Young's convolution
inequality and 
$\bigl\|\abs{\cdot} \pd_{x^i} J_\delta(\cdot)
\bigr\|_{L^1(\T^d)}\lesssim 1$. Similarly,
{
\begin{align*}
&\norm{I_2}_{L^q(\T^d)} \\
& = \norm{\int_{\T^d}  \, \pd_{x^i}  J_\delta(\cdot-y)
	\Bigl( \sigma_{ik}(y)\pd_{y^j} \sigma_{jk}(y) 
	- \sigma_{ik}(\cdot) \pd_{x^j} \sigma_{jk}(\cdot)\Bigr)
	 u(y) \;\d y}_{L^q(\T^d)}\\
	 & \le \norm{\int_{\T^d}  \,\abs{\cdot - y}\abs{ \pd_{x^i}  J_\delta(\cdot-y)}
	\frac{\abs{\sigma_{ik}(y)\pd_{y^j} \sigma_{jk}(y) 
	- \sigma_{ik}(\cdot) \pd_{x^j} \sigma_{jk}(\cdot)}}{\abs{\cdot - y}}
	 \abs{u(y)} \;\d y}_{L^q(\T^d)}\\
	 & \le C \norm{\abs{\cdot} \pd_{x^i} J_\delta}_{L^1(\T^d)} 
	 	 \norm{\nabla\bk{\sigma_{ik} \,\pd_{y^j}\sigma_{jk}}}_{L^{pq/(p - q)}}
	 	 \norm{u}_{L^p(\T^d)}.
\end{align*}

}
And finally, 
{
\begin{align*}
&	\norm{{I}_3}_{L^q(\T^d)}\\
	& =\norm{\int_\T^d \pd_{xx}^2
	J_\delta(\cdot-y)\big(2 \sigma_{jk}(\cdot)\sigma_{ik}(y)
	-\sigma_{jk}(\cdot)\sigma_{ik}(\cdot) -\sigma_{jk}(y)\sigma_{ik}(y)\big)
	 u(y) \;\d y}_{L^q(\T^d)}
	\\ &
 	\le C \norm{\int_\T^d \abs{\cdot-y}^2
 	\abs{\pd_{x^ix^j}^2J_\delta(\cdot-y)}
 	\abs{\frac{\sigma_{ik}(\cdot) - \sigma_{ik}(y)}{\cdot-y}}
 	 	\abs{\frac{\sigma_{jk}(\cdot) - \sigma_{jk}(y)}{\cdot-y}}
 	\abs{ u(y)} \;\d y}_{L^q(\T^d)}
 	\\ &
 	\le C\norm{\nabla\sigma_{ik}}_{L^{2pq/(p - q)}(\T^d)}
 		\norm{\nabla\sigma_{jk}}_{L^{2pq/(p - q)}(\T^d)}
 	\Bigl\|(\cdot)^2\pd_{x^ix^j}^2J_\delta(\cdot)\Bigr\|_{L^1(\T^d)}
 	\norm{ u}_{L^p(\T^d)}
 	\\& \le
 	C  	\norm{ u}_{L^p(\T^d)}.
\end{align*}}
We also have
$$
\norm{\eqref{eq:commute5}}_{L^q(\T^d)}
\le C \norm{J_\delta}_{L^1(\T^d)}
\norm{\pd_{x^i}(\sigma_{ik}\pd_{x^j} \sigma_{jk})}_{L^{pq/(p - q)}(\T^d)}
\norm{ u(t)}_{L^p(\T^d)}.
$$

Given the last three ($\delta$-independent) bounds,
it is sufficient to establish convergence of
\eqref{eq:commutators} under the
assumption that $\ss$, $ u$ are smooth (in $x$).
The general case follows by density using the
established bounds. Under this assumption, 
and using $\pd_{x^ix^j}J_\delta = \pd_{x^jx^i}J_\delta$ 
we have
\begin{align*}
	{I}_3
	& =\int_{\T^d} \pd_{x^ix^j}^2 J_\delta(x - y)
	\bigl( 2 \sigma_{ik}(x)\sigma_{jk}(y)-\sigma_{ik}(x) \sigma_{jk}(x)
	-\sigma_{ik}(y) \sigma_{jk}(y)\bigr)  u(y) \;\d y
		\\ &
	=- 2\int_{\T^d}\pd_{x^ix^j}^2 J_\delta(x-y)\\
&\qquad\times\bk{\bk{\sigma_{ik}(y) - \sigma_{ik}(x)} \sigma_{jk}(y)
	 + \bk{\sigma_{jk}(x) - \sigma_{jk}(y)}\sigma_{ik}(x)}
	  u(y) \;\d y
	\\ &
	=- 2\int_{\T^d}\pd_{x^ix^j}^2 J_\delta(x-y)\\
&\qquad\times\bk{\bk{\sigma_{ik}(y) - \sigma_{ik}(x)} \sigma_{jk}(y) 
	+ \bk{\sigma_{ik}(x) - \sigma_{ik}(y)}\sigma_{jk}(x)}
	  u(y) \;\d y
	\\ &
	=- 2\int_{\T^d} u(y) \pd_{x^ix^j}^2 J_\delta(x-y)\\
&\quad\times\frac{\bk{x - y}\otimes\bk{x - y}}{2}:
	\frac{\sigma_{ik}(y)-\sigma_{ik}(x)}{\abs{y-x}^2}(y -x) \otimes\frac{\sigma_{jk}(y)-\sigma_{jk}(x)}{\abs{y-x}^2} (y - x) 
	   \;\d y
	\\ & = -2 u(x) {\rm grad}(\sigma_{ik}) \otimes {\rm grad}(\sigma_{jk})
	 : \int_{\T^d} \frac{z \otimes z}{2}\pd_{z^iz^j}^2
	J_\delta(z) \; \d z+o_\delta(1),
\end{align*}
where $\int_{\T^d} \frac{z \otimes z}{2}
\pd_{z^iz^j}^2 J_\delta(z) \; \d z 
=  (e_i \otimes e_j + e_j \otimes e_i)/2$, 
with $e_i$ being the unit vector in 
the $i$th direction. A similar
calculation can be carried out for ${I}_1$, 
$I_2$, and \eqref{eq:commute5}
in each of which case there 
is only one derivative on the mollifier, 
and can be treated as in the 
proof of \cite[Lemma II.1]{DL1989}.
Reasoning as in the proof
of \cite[Lemma II.1]{DL1989}, we arrive at
\begin{align*}
	{I}_1&\overset{\delta \downarrow 0}{\to}
	2\abs{\nabla \ss}^2  u,
	 \qquad \qquad \quad \qquad \qquad\,\,  I_2 \overset{\delta \downarrow 0}{\to} 
	 	\pd_j \bk{\sigma_{ik} \,\pd_i \sigma_{jk}} u,\\
	\quad {I}_3  &\overset{\delta \downarrow 0}{\to}
-\bk{	\pd_j \sigma_{ik} \pd_i \sigma_{jk} +  \abs{\nabla \ss}^2}u,\quad\,\,
	-\eqref{eq:commute5}\overset{\delta \downarrow 0}{\to}
	-\nabla \bigl(\ss\nabla \ss \bigr)  u,
\end{align*}
in $L^q(\R)$, for
	$\d\mathbb{P}\otimes \d t$-a.e. 
Adding these terms together, with
reference to \eqref{eq:commutator_decomposition}, 
$$
\Bigl[\bigl[\KK,\JJ_\delta\bigr],\KK\Bigr]( u) 
	= I_1 + I_2  + I_3 - \eqref{eq:commute5} \to  0, \quad \text{in $L^q(\R)$, for
	$\d\mathbb{P}\otimes \d t$-a.e.},
$$ 
and
an application of the dominated convergence 
theorem establishes the lemma.
\end{proof}

\bibliographystyle{plain}

\end{document}